\documentclass{article}
\usepackage{mathrsfs}
\usepackage{graphicx}
\pdfoutput=1
\usepackage{setspace}

\usepackage{amsmath, amsfonts,amsthm, amssymb}
\usepackage{tensor}
\usepackage{appendix}
\usepackage{mathtools}
\usepackage{appendix}
\usepackage[light,math]{anttor}
\usepackage[T1]{fontenc}
\usepackage{hyperref}
\hypersetup{urlcolor=blue, colorlinks=true}
\usepackage{fancyhdr}
\usepackage[square, numbers, comma, sort&compress]{natbib}
\usepackage[nottoc,numbib]{tocbibind}

\newcommand{\rel}{\mathbb{R}}

\DeclareMathOperator{\de}{d}
\newcommand{\dt}[1]{\frac{\de\!#1 }{\de\! t}}
\newcommand{\ds}[1]{\frac{\de\!#1 }{\de\! s}}
\newcommand{\RN}{\overline{R}}

\newcommand{\nablaN}{\overline{\nabla}}
\newcommand{\RicN}{\overline{\Ric}}

\newcommand{\GammaN}{\overline{\varGamma}}

\DeclareMathOperator{\dtm}{det}

\DeclareMathOperator{\trc}{trace}
\DeclareMathOperator{\grp}{graph}

\DeclareMathOperator{\grph}{graph}

\DeclareMathOperator{\Ric}{Ric}

\DeclareMathOperator{\Id}{\bf{Id}}
\DeclareMathOperator{\e}{e}

\newtheorem{thm}{Theorem}[section]
\newtheorem{prop}[thm]{Proposition}

\newtheorem{lem}[thm]{Lemma}

\newtheorem{cor}[thm]{Corollary}

\theoremstyle{definition}

\newtheorem{example}[thm]{Example}
\newtheorem{remark}[thm]{Remark}

\numberwithin{equation}{section}

\begin{document}

\pagenumbering{roman}
\title{On the inverse mean curvature flow in warped product manifolds}
\date{}
\author{Thomas Mullins} 
\maketitle

\begin{abstract}
We consider the warped product manifold, $\rel_+ \times_{\Id} M^n$, with Riemannian metric $\gamma\equiv \de r^2 \oplus r^2 \sigma$, where $(M^n, \sigma)$ is a smooth closed Riemannian $n$-manifold. We investigate what sufficient curvature condition is required of $\sigma$ to ensure that a solution to the inverse mean curvature flow - commencing with a surface described as the graph of a global $C^{2,\alpha}$ function on $M^n$ - exists for all times $t>0$.
\end{abstract}

\setcounter{page}{0} 
\pagenumbering{arabic}
\tableofcontents

\section{introduction}
The general outward curvature flow (of which the inverse mean curvature flow is a special case) was studied by Gerhardt \cite{ger1} and Urbas \cite{urbas} in $\rel^{n+1}$ for a {\em starshaped} initial surface $S_0$, which is equivalent to there existing a $u_0:S^n \rightarrow \rel_+$ such that
\[
	S_0 = \grp u_0,
\]
where $S^n$ is the unit $n$-sphere. In both papers, the idea was to use polar coordinates to describe the flow of the surfaces and prove existence for all times $t>0$. This is equivalent (see Example \ref{polar coords}) to the flow of surfaces in the warped product $\rel_+ \times_{\Id} S^n$. So it was a natural question to ask: under what conditions on an arbitrary manifold $M^n$ will a solution to this flow in the warped product $\rel_+ \times_{\Id}M^n$ also exist for all times $t>0$?\\
\\
Indeed, starting out with a sufficiently smooth function $u_0$ from the compact $n$-manifold $M^n$ into the positive real numbers, defining an embedded hypersurface $M_0$ in the warped product $\rel_+ \times_{\Id} M^n$ in the obvious manner
\[
	M_0 \equiv \grp u_0 = \{ (u_0(p),p), p\in M^n \},
\] 
we investigate the curvature conditions required of $M^n$. We will show that a positive definite Ricci tensor suffices for the solution to indeed exist for all times $t>0$ and remains graphical thoughout the evolutionary process. We also show, via a standard rescaling method, that the solution, in a certain sense, is asymptotically `$\{\infty\} \times M^n$'. We assert the main theorem:
\begin{thm}[Main Theorem]
	Let $(M^n,\sigma)$ be a smooth closed Riemannian manifold of dimension $n$, whose curvature satisfies
\begin{equation} \label{ricci bound}
	\Ric_M(X) > 0
\end{equation}
for all $X\in TM$. Let $u_0:M^n \rightarrow \rel_+$ be of class $C^{2,\alpha}$ whose corresponding embedding
\[
	x_0:M^n \rightarrow N^{n+1}, \quad p \mapsto (u_0(p),p),
\] 
has strictly positive mean curvature $H_0$. Then the evolution equation
\begin{equation}\label{eq}\left\{
  \begin{array}{l l}
    \dot{x}  & = H^{-1}\nu \\
   x_0 & =  x(0,\cdot),
  \end{array} \right.\end{equation}
where $\nu$ is the outward pointing unit normal to the evolving hypersurfaces $M_t \equiv x(t,M)$, and $H$ is the mean curvature of $M_t$, has a unique solution of class $H^{2+\beta,\frac{2+\beta}{2}}(Q_{\infty}) \cap H^{2+m+\gamma,\frac{2+m+\gamma}{2}}(Q_{\epsilon,\infty})$ for any $0<\beta < \alpha$, $\gamma \in (0,1)$, $\epsilon > 0$ and $m\geq 0$ that exists for all times $t>0$. The rescaled surfaces
\[\tilde{M}_t = \tilde{x}(t,M) = \grph \tilde{u}(t,\cdot), \]
where $\tilde{u} \equiv u\e^{-t/n}$, converge exponentially fast to a constant embedding \[ M^n \xhookrightarrow{r_{\infty}}\lbrace r_{\infty} \rbrace\times  M^n ,\]
with
\begin{equation*}
	r_{\infty} = \left[ \frac{|M_0|}{|M^n|}\right]^{1/n}.
\end{equation*}
\end{thm}

\subsection{notation and conventions}\label{not and conv}
We are working primarily with a family of embeddings of an arbitrary $n$-manifold with a fixed metric $(M^n,\sigma)$ into an $(n+1)$-manifold also equipped with a fixed metric $(N^{n+1},\gamma)$. 
Coordinates on $N^{n+1}$ will be denoted by $\lbrace x^a \rbrace$ and on $M^n$ as $\lbrace y^i \rbrace$, where the indices from the set $\lbrace a,b,c,d,e \rbrace$ run from $0$ to $n$ and the indices of the set $\lbrace i,j,k,l,m,r,s \rbrace$ run from $1$ to $n$ throughout this paper. The indices will be raised and lowered (corresponding to contravariant and covariant quantities respectively) with the corresponding metric as is usual practise.\\
\\
With $g=\{g_{ij}\}$ we denote the induced metric on the embedded manifolds, $M_t$. Although $g=g(t)$, we suppress this dependence on $t$ to spare the reader a debauch of indices and arguments. We should, however, bear this dependence in mind. The covariant differential operators on $M^n$, $M_t$ and $N^{n+1}$ are signified either by $D=D(\sigma)$, $\nabla = \nabla(g)$ and $\nablaN = \nablaN(\gamma)$; or simply by subscripted indices, from the set from $1$ to $n$ for $M^n$ and $M_t$ (it will always be clear from the context whether we are on $M^n$ or $M_t$), alternatively from the set of indices running from $0$ to $n$ for covariant differentiation on $N^{n+1}$. In situations of potential ambiguity there will occasionally be a subscripted semicolon `$_;$' in front of the index to be covariantly differentiated. For the partial derivative, we either replace the subscripted semicolon with a subscripted comma, as is common in the literature, or include the partial derivative symbol $\partial$ with appropriate subscripted indices.\\
\\
We will mainly use tensor notation to describe tensor quantities on $M^n$, $M_t$ and $N^{n+1}$. For instance if $w$ is a function from $N^{n+1}$ to $\rel$ then $w_{;\,ab}$ (or simply $w_{ab}$) would represent the Hessian of $w$. For an arbitrary tensor $\mathcal{T}$ of valence $[ ^k _l ]$, we define the covariant derivative of that tensor $\nabla \mathcal{T}$ to be the $[^k _{l+1}]$-tensor by 
\begin{equation}\label{tensor covariant derivative} 
	\begin{split}
	(\nabla \mathcal{T})&(\omega_1,\dots,\omega_k,X^1,\dots,X^l,Z) = (\nabla_Z \mathcal{T})(\omega_1,\dots,\omega_k,X_1,\dots,X_l)\\
			\equiv & Z\left(\mathcal{T}(\omega_1,\dots,\omega_k,X^1,\dots,X^l) \right) \\
		& + \sum_p \mathcal{T}(\omega_1,\dots,\nabla_Z \omega_p,\dots, \omega_k,X^1,\dots,X^l) \\
			& - \sum_q \mathcal{T}(\omega_1,\dots,\omega_k,X^1,\dots,\nabla_Z X^q,\dots,X^l).
	\end{split}
\end{equation}
Here, the $\omega_j$ are covectors and the $X^j$ are vectors. This can be expressed more concisely in index notation as follows 
\begin{equation*} 
\mathcal{T}^{i_1 ... i_k} _{j_1 ... j_l;p} = \mathcal{T}^{i_1 ... i_k} _{j_1 ... j_l,p} + \sum_{q=1}^k \tensor{\varGamma}{^{i_q}_p_e}\mathcal{T}^{i_1 ..e.. i_k} _{j_1 ... j_l} - \sum_{r=1}^l \tensor{\varGamma}{^e_p_{j_r}} \mathcal{T}^{i_1 ... i_k} _{j_1 ..e.. j_l},
\end{equation*}
where $\tensor{\varGamma}{^c_a_b}$ are the connection coefficients. The Einstein summation convention will be used throughout.\\
\\
The norm of a tensor $\tensor{\mathcal{T}}{^k_i_j}$ of valence $\left[ ^1_2 \right]$ with respect to the metric $\sigma$ is defined by
\begin{equation*}
	|\mathcal{T}|^2_{\sigma} \equiv \sigma_{ij} \sigma^{kl}\sigma^{rs} \tensor{\mathcal{T}}{^i_k_r}\tensor{\mathcal{T}}{^j_l_s}.
\end{equation*}
The norms with respect to $g$ and $\gamma$ are defined analogously, as is the case for a tensor of arbitrary valence.\\
\\
At times we will need to distinguish between the connection coefficients and other tensor quantities (for instance, curvature) on $(M^n,\sigma)$, on $(M_t,g)$ and on $(N^{n+1},\gamma)$. We shall do so by a superscripted index of the metric for $M^n$ and $M_t$, and a bar for $N^{n+1}$, for example the Christoffel symbols of the connections ${^{\sigma}\tensor{\varGamma}{^k_i_j}}$, ${^g\tensor{\varGamma}{^k_i_j}}$ and $\tensor{\GammaN}{^c_a_b}$ for the corresponding quantities on $M^n$, $M_t$ and $N^{n+1}$ respectively.\\
\\
We adopt the standard convention for the commutator of the covariant derivative (the Riemann curvature tensor), corresponding to
\begin{equation} \label{commutator vec} 
\begin{split}
	R(U,V)X \equiv & \mathop{\nabla}\limits_{U}\mathop{\nabla}\limits_{V} X - \mathop{\nabla}\limits_{V}\mathop{\nabla}\limits_{U} X - \mathop{\nabla}\limits_{[U,V]} X \\
	R(U,V,X,Z) \equiv & \langle R(U,V)X,Z \rangle,
\end{split}
\end{equation} 
for tangent vectors $U,V$ and vector fields $X,Z$. In index notation, this can be expressed for a contravariant quantity as
\begin{equation}\label{curvature index notation}
	\nabla_a \nabla_b X^c - \nabla_b \nabla_a X^c = \tensor{R}{_a_b^c_d} X^d,
\end{equation}
and for a covariant quantity as
\begin{equation}\label{commutator form}
	\nabla_a \nabla_b \varphi_c - \nabla_b \nabla_a \varphi_c =\tensor{R}{_a_b_c^d} \varphi_d.
\end{equation}
The Ricci curvature, as remarked in the introduction, is the trace of the Riemann curvature tensor. Formally
\begin{equation}\label{ricci}
	\Ric(U,V) \equiv \sum^{n}_{i=1} R(U,E_i,V,E_i), 
\end{equation}
where $\{E_i\}_{i=1,...,n}$ is an orthonormal frame. Equivalently, in index notation can be expressed as
\begin{equation*}
	\Ric(U,V) = \tensor{R}{_i_k}U^i V^k \equiv \tensor{R}{_i_m_k^m}U^i V^k.
\end{equation*}
We can also define the quadratic form of the Ricci tensor, a notation which will at times prove useful:
\begin{equation*}
	\Ric(U) \equiv \Ric(U,U). 
\end{equation*}
\noindent The {\em function space}, $H^{2+m+\alpha,\frac{2+m+\alpha}{2}}(Q_T)$, in which our functions reside is defined in \eqref{fcn space}. The definition of the norm can be found in \cite[\S 2.5]{ger2}.
\newpage
\section{hypersurface geometry}

\subsection{the fundamental equations}
Let us assert the fundamental equations from the theory of isometric immersions. These relate the geometry of the submanifold to the geometry of the ambient space via the second fundamental form, $A\equiv \{h_{ij}\}_{i,j=1,\dots,n}$. Recall that if 
\[
	x:M^n \rightarrow N^{n+1}
\] 
is an {\em embedding} (or {\em immersion}) of $M^n$ into $N^{n+1}$, then $\de\! x_p$ has rank $n$ for all $p\in M^n$.\\
\\
A wide range of literature covers this topic including proofs of the following relations, see for instance \cite{do carmo, ger2}.
\begin{lem}[Gauss formula] 	
There exists a symmetric $\left[^0_2\right]$-tensor field $h:TM \times TM \rightarrow \rel$ that satisfies
	\begin{equation} \label{gauss formula}
		x_{ij} = - h_{ij}\nu.
	\end{equation}
\end{lem}

\begin{lem}[Weingarten equation]
	The unit normal $\nu$ satisfies the identity
		\begin{equation}\label{weingarten eq}
			\nu_i = h^k_{i}x_k.
		\end{equation}
\end{lem}

\begin{cor} The second derivative of $\nu$ satisfies
\begin{equation*}
	\nu_{ij} = h^k_{i;j}x_k - h^k_i h_{kj} \nu.
\end{equation*}
\end{cor}

\begin{lem}[Gauss equation]\label{gauss eq lem}
The Riemann curvature of the submanifold is related to the curvature of the ambient space by
	\begin{equation}\label{gauss eq}
		{^g\tensor{R}{_i_j_k_l}} = h_{ik}h_{jl} - h_{il}h_{jk} + \RN(x_i,x_j,x_k, x_l).
	\end{equation}
\end{lem}

\noindent By contracting \eqref{gauss eq} with the inverse metric $\gamma^{ab}$ of the ambient space, we obtain identities for the Ricci curvature as well as the scalar curvature of the submanifold. These we give explicitly in the following two corollaries.
\begin{cor} \label{ricci gauss}
The Ricci curvature of the submanifold is related to the curvature of the ambient space by
	\begin{equation*}
		{^g\tensor{R}{_i_k}} = \RicN(x_i,x_k) - \RN(\nu,x_i,\nu,x_k) + Hh_{ik} - h_{i m} h^m_k
	\end{equation*}
\end{cor}
\begin{cor} \label{scalar gauss}
The Scalar curvature of the submanifold is related to the curvature of the ambient space by
	\begin{equation}
		{^gR} = \RN - 2\RicN(\nu) + H^2 - |A|^2_g,
	\end{equation}
where $A\equiv \{h_{ij}\}_{i,j = 1,\dots,n}$ (a notation that we will use on occasion).
\end{cor}
\noindent Our final equation of the fundamental classical variety relates the ambient curvature to the differential of the second fundamental form.
\begin{lem}[Codazzi equation] The arguments of the $\left[^0_3\right]$-tensor $\nabla h$ commute as follows
\begin{equation}\label{codazzi}
	h_{ij;k} - h_{ki;j} = \RN(\nu, x_i, x_j, x_k).
\end{equation}
\end{lem}
\noindent We now look at how the second derivatives of the second fundamental form commute. These identities were first discovered by Simons \cite{simons} and come in very useful.
\begin{lem} The second derivatives of $h_{ij}$ satisfy
	\begin{equation} \label{hij der comm}
	\begin{split}
		h_{ij;kl} = & h_{kl;ij} +  h_{mj} h_{il}h^m_k - h_{ml}h_{ij}h^m_k + h_{mj} h_{kl} h^m_i - h_{ml} h_{kj} h^m_i\\
			& + \RN( x_k, x_i, x_l, x_m) h^m_j +\RN( x_k, x_i, x_j, x_m) h^m_l + \RN(x_m, x_i, x_j, x_l ) h^m_k \\
			& + \RN(x_m, x_k, x_j, x_l) h^m_i+ \RN(\nu, x_i, \nu, x_j) h_{kl} - \RN(\nu, x_k, \nu, x_l) h_{ij}\\
			& + (\nablaN \RN)(\nu, x_i, x_j, x_k, x_l) + (\nablaN \RN)(\nu, x_k, x_i, x_l, x_j ).
		\end{split}
	\end{equation}
\end{lem}
\begin{proof}
	Using \eqref{gauss formula}, \eqref{weingarten eq},\eqref{gauss eq}, \eqref{codazzi} and the commutator of covariant tensor quantities \eqref{commutator form} and the tensor covariant derivative \eqref{tensor covariant derivative} (and not forgetting the symmetries of $\RN$)
	\begin{align}
\notag	h_{ij;kl} = & (h_{ki:j} + \RN(\nu, x_i, x_j, x_k)_{;l} \\ 
\notag		= & h_{ki;jl} + (\nablaN\RN)(\nu, x_i, x_j x_k x_l) \\ 
\label{three}& + \RN(x_m, x_i, x_j, x_k) h^m_l - \RN(\nu, x_i, \nu, x_k) h_{lj} - \RN(\nu, x_i, x_j, \nu) h_{kl} \\ 
\notag	h_{ki;jl} = & h_{ki;lj} + {^g\tensor{R}{_m_i_j_l}}h^m_k + {^g\tensor{R}{_m_k_j_l}}h^m_i + \RN(x_m, x_i, x_j, x_l) h^m_k +  \RN(x_m, x_k, x_j, x_l) h^m_i\\ 
\notag		= & h_{ki;lj} + h_{mj} h_{il}h^m_k - h_{ml}h_{ij}h^m_k + h_{mj} h_{kl} h^m_i - h_{ml} h_{kj} h^m_i\\
\label{two}	& + \RN(x_m, x_i, x_j, x_l) h^m_k +  \RN(x_m, x_k, x_j, x_l) h^m_i\\ 
\notag	h_{ki;lj} = & h_{kl;ij} + (\nablaN\RN)(\nu, x_k, x_i, x_l, x_j) \\ 
\label{one}	&+ \RN(x_m, x_k, x_i, x_l) h^m_j - \RN(\nu, x_k, \nu, x_l) h_{ij} - \RN( \nu, x_k, x_i, \nu) h_{jl}. 
	\end{align}
Plugging \eqref{one} into \eqref{two} and \eqref{two} into \eqref{three} and using again the symmetries of the full tensor $\RN$ gives the desired result.
\end{proof}
\begin{cor}[Simons' identity]
The Laplacian of the second fundamental form satisfies
	\begin{equation}
		\begin{split}
			\Delta_{M^{\prime}} h_{ij} = & H_{ij} - |A|^2_g h_{ij} + H h_{im} h^m_j + H \RN(\nu, x_i, \nu, x_j)- \RicN(\nu) h_{ij} \\
					& + g^{kl}\big[ \RN( x_k, x_i, x_l, x_m) h^m_j  +\RN( x_k, x_i, x_j, x_m) h^m_l + \RN(x_m, x_i, x_j, x_l ) h^m_k \\
					& + \RN(x_m, x_k, x_j, x_l) h^m_i + (\nablaN_{x_l} \RN)(\nu, x_i, x_j, x_k) + (\nablaN_{ x_j } \RN)(\nu, x_k, x_i, x_l)\big].
		\end{split}
	\end{equation}
\end{cor}
\begin{proof}
	Contract \eqref{hij der comm} with $g^{kl}$ and use the (anti)symmetric properties of the full tensor $\RN$.
\end{proof}

\subsection{geometry of the warped product} \label{wp geometry}
The {\em warped product}, $P\times_{h} S$, of two Riemannian manifolds, $(P^p,\varpi)$ and $(S^s,\varsigma)$, with {\em warping factor} $h$ is the product manifold, $P\times S$, equipped with the metric
\[
	\gamma^{P\times_{h} S} = \varpi \oplus  h^2 \varsigma,
\] 
In this manner, providing $h$ is not degenerate, $(P\times_{h} S, \gamma^{P\times_{h} S})$ also becomes a Riemannian manifold.\\
\\
In this section we look at some properties of warped product spaces in a more specialised setting. We consider the warped product $I\times_f M^n$, where $I$ is an arbitrary one dimensional manifold parameterised with $r$, $(M^n,\sigma)$ is a Riemannian $n$-manifold with metric $\sigma$, and $f$ is a positive monotonically increasing (preferably differentiable) function of $r$. Thus we have the metric 
\begin{equation}  \label{wp metric}
	\gamma =\de\!r^2 \oplus f^2 \sigma.
\end{equation} 
Indeed, this metric is conformal to a metric of the form 
\begin{equation}\tag{*}\label{metric}
	 \de\!\rho^2 \oplus \rho^2 \sigma, 
\end{equation} with conformal factor $\phi(\rho)$ where $\de \! r = \phi(\rho)\de \!\rho$ and $f(r) = \rho \phi(\rho)$,  (see \cite{petersen}).\\
\\
During the course of this work, we concern ourselves only with metrics of the form \eqref{metric}, or more accurately, with the warped product $\rel_+ \times_{\Id} M^n$. \\
\\
The arbitrarily chosen coordinates $\{ y^i \}$ of $M^n$ are extended via the radial function $r$ to coordinates $\{y^a\}$ of $N=\rel_+ \times_{\Id} M^n$ whereby
\[
	y^0 \equiv r.
\]
In this manner, the basis $\{ e_i \equiv \partial / \partial y^i\}_{i=1,\dots,n}$ of $TM$ is extended to a basis, $\{ e_a \}_{a=0,\dots,n}$, of $TN$. This allows for convenient indexing in what follows. Moreover, tensor quantities on $N^{n+1}$ will be dealt with in this basis, for example
\[
	\tensor{\RN}{_a_b_c_d} \equiv \RN(e_a,e_b,e_c,e_d).
\]
The level sets, $\lbrace r \rbrace \times M^n\equiv \lbrace r = \text{const} \rbrace$, of the radial function $r$ of the warped product, have the induced metric 
\begin{equation}\label{metric alpha}
	\alpha_{ij} = r^2\sigma_{ij},
\end{equation}
and the second fundamental form, $\beta_{ij},$ of the inclusion $M^n \xhookrightarrow{r} \lbrace r \rbrace \times  M^n $ is easily shown to be \begin{equation}\label{sff beta}
	\beta_{ij} = r\sigma_{ij}.
\end{equation}
\noindent We wish to express the connection on $N^{n+1}$ in terms of quantities on $M^n$ and the radial function $r$, whereby we consider surfaces  $\lbrace r \rbrace \times M^n$. Using the well known formula for the Levi-Civita connection coefficients (see for example \cite{jost})
\begin{equation}\label{christoffel define}
\tensor{\GammaN}{^c_a_b} = \frac{1}{2} \gamma^{c e} (\partial_{a} \gamma_{e b} + \partial_{b} \gamma_{a e} - \partial_{e} \gamma_{a b}),
\end{equation}
where $\partial_a \equiv \partial / \partial y^a$, the following fact almost derives itself
\begin{equation}\label{christoffel}
\tensor{\GammaN}{^c_a_b} = \left\{
  	\begin{array}{l l}
    		\tensor{\varGamma}{^k_i_j} 	&	\quad 		\text{if $(a,b,c)=(i,j,k)$}\\
    		-r\sigma_{ij}				& 	\quad 		\text{if $c=0, (a,b) =(i,j)$ }\\
		r^{-1} \delta^k_j  			& 	\quad 		\text{if $a=0, (b,c) =(j,k) \vee b=0, (a,c) = (j,k)$ }\\
		0 & \quad \text{otherwise}.
  \end{array} \right.
\end{equation}
\begin{lem}
	The Riemann curvature tensor of the warped product $\rel_+ \times_{\Id} M^n$ vanishes in the radial direction, i.e.
 \begin{equation}\label{curvature}
	\tensor{\RN}{_0_b_c_d} = 0.
\end{equation}
\end{lem}

\begin{proof}
It is easily derived from \eqref{curvature index notation} that the curvature tensor in terms of $\tensor{\GammaN}{^c_a_b}$ (see \cite[\S 3]{wald}) can be expressed as
\begin{equation}\label{curvature gamma}
	\tensor{\RN}{_a_b_c^d} = \partial_b \tensor{\GammaN}{^d_a_c} - \partial_a \tensor{\GammaN}{^d_b_c} + \tensor{\GammaN}{^d_b_e}\tensor{\GammaN}{^e_a_c} - \tensor{\GammaN}{^d_a_e}\tensor{\GammaN}{^e_b_c}.
\end{equation}
Combining \eqref{christoffel}  and \eqref{curvature gamma} we deduce (sparing the details)
\begin{equation}
		\tensor{\RN}{_0_j_k^l} = \tensor{\RN}{_0_j_k^0} =0.
\end{equation}
Lowering the final index with $\gamma_{ab}$ shows
\[
\tensor{\RN}{_0_i_j_a} = \gamma_{ab} \tensor{\RN}{_0_i_j^b} = 0.
\]
The symmetries of $\RN$ lead to the result.
\end{proof}
\begin{lem}
	The embedding $M^n \xhookrightarrow{r}  \{r\}\times M^n$ has constant mean curvature of $r^{-1}n.$
\end{lem}
\begin{proof}
	Contracting the second fundamental form $\beta=r\sigma$ with the inverse induced metric $\alpha^{-1} = r^{-2}\sigma^{-1}$ gives \[ H = \alpha^{ij} \beta_{ij} = r^{-1} \sigma^{ij} \sigma_{ij} = r^{-1}n.\]
\end{proof}
\begin{prop}
	\begin{equation}
		\Ric_M \geq 0 \iff \RicN \geq (1-n) \sigma
	\end{equation}
\end{prop}
\begin{proof}
	Once again using \eqref{christoffel} and \eqref{curvature gamma}, we deduce
\begin{equation}\label{riemann level set}
	{^{\alpha}\tensor{R}{_i_j_k^l}} = {^{\sigma}\tensor{R}{_i_j_k^l}},
\end{equation}
where ${^{\alpha}\tensor{R}{_i_j_k^l}}$ is the Riemann curvature tensor with respect to the metric $\alpha_{ij}$ from \eqref{metric alpha} on the level set $\{ r \} \times M^n$. Taking the trace over $j$ and $l$ shows
\[
	\Ric_{\{r\}\times M} = \Ric_M.
\]
Using \eqref{metric alpha} and \eqref{sff beta} 
\[
	\beta^i_j = \alpha^{im}\beta_{mj} = r^{-1}\delta^i_j.
\]
Now apply Corollary \ref{ricci gauss}, substituting $\beta$ for $h$ and taking note of \eqref{curvature}, to reveal
\begin{equation} \label{ric m}
	\Ric_M = \RicN + (n-1)\sigma.
\end{equation}
\begin{remark}
	If we contract \eqref{riemann level set} with $\alpha_{lm}$, we get
	\begin{equation}
		{^{\alpha}\tensor{R}{_i_j_k_m}} = \alpha_{lm}\left[{^{\sigma}\tensor{R}{_i_j_k^l}}\right] = r^2 \sigma_{lm}\left[ {^{\sigma}\tensor{R}{_i_j_k^l}}\right] = r^2 \left[ {^{\sigma}\tensor{R}{_i_j_k_m}} \right].
	\end{equation}
This equality shows that as the embedded hypersurface varies in the radial direction with equal magnitude at every point (i.e. rescales), the curvature of the rescaled hypersurface is inversely proportional to the radial distance. Intuitively this can be seen if we consider a family of expanding 2-spheres in $\rel^{3}$.
\end{remark}
\end{proof}
\begin{example}[polar coordinates]\label{polar coords}
	The archetypal example of a warped product manifold is $\rel^{n+1} \setminus \{0\}$ equipped with the canonical metric $\delta^i_j$ expressed as
	\begin{equation*}
		\rel_+ \times_{\Id} S^n
	\end{equation*}
	with the metric as in \eqref{wp metric} where $\sigma$ is now a metric on $S^n$. This corresponds simply to expressing elements of $\rel^{n+1}\setminus \{{\bf 0}\}$ in polar coordinates. As aforementioned, this is how Gerhardt and Urbas tackled the issue of the outward flow of an embedded surface in Euclidean space in \cite{ger1} and \cite{urbas} respectively.
\end{example}


\subsection{hypersurfaces as graphs of functions}
The theory of general hypersurfaces (see eg. \cite[\S 6]{do carmo}) can be easily applied to the special case in which the ambient space is the warped product $N=\rel_+ \times_{\Id} M^n$ from \S \ref{wp geometry}, where the metric $\gamma$ of $N^{n+1}$ in terms of the radial function $r$ and the metric $\sigma$ of $M^n$, is expressed as
\[
	\gamma = \de\! r^2 \oplus r^2 \sigma
\] and the embedding $x$ can be seen as the graph of a function $u:M^n\rightarrow \rel_+$, however, not all submanifolds of $N^{n+1}$ can be expressed as such.\\
\\
In this respect we obtain the induced metric $g$:
\begin{equation} \label{g def}
	g_{ij}(p) \equiv g_p(e_i, e_j) \equiv (x^{*}\gamma)_p(e_i, e_j) \equiv \gamma_{x(p)=(u(p),p)}(x_{*}(e_i), x_{*}(e_j)) 
\end{equation} 
where 
\begin{equation}\label{xi}
	x_{*}(e_i) = \de\!x(e_i) = u_ie_0 + e_i \equiv x_i,
\end{equation}
and $e_i = \partial / \partial y^i$. Thus, we find
\[
	g_{ij} = \gamma(x_i, x_j) = u_iu_j + u^2 \sigma_{ij} = u^2(\sigma_{ij} + \varphi_i \varphi_j),
\]
where $\varphi \equiv \log u$ throughout. Since $u$ takes solely positive values, $\varphi$ is always well defined. The basis $\{ e_i = \partial / \partial y^i \}_{i=1,...,n}$ of $TM$ gives rise to a basis of $Tx(M)$ via the differential (or {\em push forward}, see previous section). These are indeed the vectors $\{ x_i \}_{i=1,...,n}$ of \eqref{xi}. \\
\\
The outward pointing unit normal vector to $x(M)$, which necessarily satisfies $\gamma(\nu, x_i) = 0$ for all $i=1,\dots,n$, is consequently given by 
\[
	\nu = (\nu^a) = \gamma^{ab} \nu_{b}
\] where the covector $\nu_{a} $ has the components $(1, -Du)$ and the normalisation factor $v^{-1}$ where 
\begin{equation}\label{v}
	v^2 \equiv 1+u^{-2}|Du|^2_{\sigma}= 1+|D\varphi|^2_{\sigma}.
\end{equation}
Thus, 
\begin{equation} \label{norm}
	\nu = (\nu^{a}) = v^{-1}(e_0 - u^{-2} u^k e_k ).
\end{equation}
It is readily checked that this vector is indeed orthogonal to each $x_i$ with respect to $\gamma$. The inverse of the induced metric is 
\[
	g^{ij} = u^{-2}(\sigma^{ij} - \varphi^i \varphi^j / v^2),
\] where $\varphi^i \equiv \sigma^{ij}\varphi_j = \sigma^{ij}D_j\varphi$ and with a quick calculation, is easily seen to be the inverse of $g_{ij}$.
\begin{prop}\label{conn g conn sigma}
	The connection of $g$ is related to the connection of $\sigma$ by
	\begin{equation} \label{conn g conn sigma eq}
	\begin{split}
		\left( \mathop{\nabla}\limits_X - \mathop{D}\limits_X \right) Y = & \sigma(D\varphi,X)Y + \sigma(D\varphi, Y)X \\
							& + \left( D^2\varphi(X,Y) - [uv]^{-2}g(X,Y) \right) D\varphi,
	\end{split}
	\end{equation} 
for $X,Y \in TM$. Equivalently
	\begin{equation}
		{^g\tensor{\varGamma}{^k_i_j}} = {^{\sigma}\tensor{\varGamma}{^k_i_j}}+ \varphi_i \delta^k_j + \varphi_j \delta^k_i - u^{-2}v^{-2}\varphi^kg_{ij}  + v^{-2}\varphi^k D_{ij}\varphi.
	\end{equation}  
\end{prop}
\begin{proof}
	Using the formula \eqref{christoffel define}, we see
	\begin{equation}
		\begin{split}
			{^g\tensor{\varGamma}{^k_i_j}}  = & \frac{1}{2}g^{kl} (g_{lj,i} + g_{il,j} - g_{ij,l}) \\
			 = & \frac{1}{2}g^{kl}\big\lbrace 2u^{-1}u_ig_{lj} + u^2\sigma_{lj,i} + u^2(\varphi_l \varphi_j)_{,i} \\
			&\hspace{35pt} + 2u^{-1}u_jg_{il} + u^2\sigma_{il,j} + u^2(\varphi_i \varphi_l)_{,j}\\
			&\hspace{35pt} - 2u^{-1}u_lg_{ij} - u^2\sigma_{ij,l} - u^2(\varphi_i \varphi_j)_{,l} \big\rbrace\\
			= & \varphi_i \delta^k_j + \varphi_j \delta^k_i - \varphi_l g^{kl}g_{ij} + {^{\sigma}\tensor{\varGamma}{^k_i_j}} \\
			& - \varphi^k \varphi^l v^{-2}\sigma_{ml} {^{\sigma}\tensor{\varGamma}{^m_i_j}} - \varphi^k \varphi_{j,i} - v^{-2}|D\varphi|^2_{\sigma} \varphi^k\varphi_{j,i}.
		\end{split}
	\end{equation}
Using the facts \[1-|D\varphi|^2 v^{-2} = v^{-2}\] and \[D_{ij} \varphi = \varphi_{j,i} - {^{\sigma}\tensor{\varGamma}{^m_i_j}}\varphi_m\] leads to the sought after result.
\end{proof}
\begin{remark}
	Notice how the right hand side of \eqref{conn g conn sigma eq} is tensorial in both $X$ and $Y$ when contracted with a covector. This is consistent with the difference of two torsion free connections being a tensor of valence $\left[ ^1_2 \right]$ (in stark contrast to the expression $\omega(\mathop{\nabla}\limits_X Y)$, for covector $\omega$, not being tensorial in $Y$ for an arbitrary connection).
\end{remark}
\begin{lem}
The second fundamental form, $A \equiv\{h_{ij} \}$, of the embedding $x$ is given as
\begin{equation}\label{sff eq}
	h_{ij} = uv^{-1}(\sigma_{ij}+\varphi_i \varphi_j - \varphi_{ij}),
\end{equation} and the \emph{Weingarten} map is given as
\[h^i_j = \left[ uv \right]^{-1}\left( \delta^i_j + ( -\sigma^{ik} + \varphi^i \varphi^k v^{-2} ) \varphi _{kj}\right).\] 
\end{lem}
\begin{proof}
Since $x_i$ is covariant with respect to $g$ as well as contravariant with respect to $\gamma$ (see proof of Lemma \ref{gauss eq lem}), in the $\{ e_a \}_{a=0,...,n}$ basis we have, 
\[
	x^a_{ij} = \partial^2_{ij}x^{a} + \tensor{\GammaN}{^a_b_d} x_i^{b} x_j^{d} - {^g\tensor{\varGamma}{^k_i_j}}x^{a}_k.
\] 
and plugging this into the Gauss formula \eqref{gauss formula}
\[
	h_{ij} = - \gamma((\partial^2_{ij}x^{a} + \tensor{\GammaN}{^a_b_d} x_i^{b} x_j^{d} - {^g\tensor{\varGamma}{^k_i_j}}x^{a}_k)e_{a}, \nu^{b}e_{b} ).
\]
The partial (not covariant) derivative $x_{,ij} = u_{,ij} e_0$, so the first term is \[ \gamma( x_{,ij}, \nu ) = v^{-1}u_{,ij}.\] The final term disappears due to orthogonality so we look at the middle term. We notice that $\gamma( e_k, \nu )= -v^{-1} u_k$ and
$\gamma( e_0, \nu ) = v^{-1}.$ Thus
\begin{equation}
	h_{ij} =-v^{-1}( u_{,ij} +  \tensor{\GammaN}{^0_b_d}x_i^{b} x_j^{d} - \tensor{\GammaN}{^k_b_d}x_i^{b} x_j^{d} u_k). 
\end{equation}
With a little computation using \eqref{christoffel} and \eqref{xi}, 
\begin{equation}\label{christoffel graph}
	\tensor{\GammaN}{^c_a_b}x_i^a x_j^b = \left\{
  		\begin{array}{l l}
    			{^{\sigma}\tensor{\varGamma}{^k_i_j}}	& \quad \text{if $(a,b,c)=(i,j,k)$}\\
    			-u\sigma_{ij}					& \quad \text{if $c=0, (a,b) =(i,j)$ }\\
			r^{-1} u_i u_j 				& \quad \text{if $a=0, (b,c) =(j,k) \vee b=0, (a,c) = (j,k)$ }\\
			0 						& \quad \text{otherwise}
  		\end{array} \right.
\end{equation}
Putting it all together we have
\begin{align*} 
	h_{ij} & = -v^{-1}\partial^2_{ij} u - v^{-1}(-u\sigma_{ij}) + v^{-1}({^{\sigma}\tensor{\varGamma}{^k_i_j}}u_k + 2u^{-1}u_iu_j)\\
		& = uv^{-1}(\sigma_{ij} + \varphi_i \varphi_j - \varphi_{ij}),
\end{align*}
where we have used the fact that $u_{ij} = \partial^2_{ij} u - {^{\sigma}\tensor{\varGamma}{^k_i_j}}u_k$ (and seeing as this is a torsion free connection, $u_{ij} = u_{ji}$).\\
\\
For the second assertion, simply compute $h^i_j = g^{ik} h_{kj}$.
\end{proof}
\begin{cor}\label{smc} The mean curvature of the embedding $x:p\mapsto (u(p),p)$ can be expressed as follows
	\begin{equation} 
		\begin{split}
		H & = [uv]^{-1} \left( n- \Delta_{\sigma} \varphi + D^2 \varphi( D\varphi, D\varphi )v^{-2}  \right) \\
			&= [uv]^{-1} \left( n- \sigma^{ij}\varphi_{ij}+ \varphi^i \varphi^j\varphi_{ij} v^{-2} )  \right)\\
			& =  [uv]^{-1} \left(  n - u^2 \Delta_g \varphi \right). 
		\end{split}
	\end{equation}
\end{cor}
\begin{remark}
Taking the $0^{\text{th}}$ component of \eqref{gauss formula} into consideration, achieves 
\begin{equation}
	\begin{split}
		-v^{-1}h_{ij} = & u_{,ij} - {^g\tensor{\varGamma}{^k_i_j}}u_k + \tensor{\GammaN}{^0_a_b}x_i^a x_j^b\\
				= & \nabla_j u_i - u\sigma_{ij}\\
				= & \nabla_j u_i - \beta_{ij},
	\end{split}
\end{equation}
where $\beta$ is the second fundamental form as in \eqref{sff beta} of the level set $\{u(p)\} \times M^n$ and our computations are being carried out at point $p\in M^n$. This can be verified with an elaborate calculation using Proposition \ref{conn g conn sigma}, the details of which, we spare the reader.
\end{remark}
\newpage
\section{inverse mean curvature flow for graphs} \label{imcf for graphs}
In \S \ref{posing the problem} we assume existence of a solution to \eqref{eq} on a maximal time interval $[0,T^{\star})$ and adapt this initial problem to the situation at hand to acquire a nonlinear {\em scalar} parabolic equation that is more manageable to work with. This fact is strengthened in \S \ref{first order estimates} where we prove existence of a solution for such an equation. We acquire two forms of this equation. One for the embedding function
\[
	u:I\times M^n \longrightarrow \rel_+
\]
and the other for the natural logarithm of $u,$
\[
	 \varphi \equiv \log u:I\times M^n \longrightarrow \rel.
\]
The advantage of the equation for $\varphi,$ as we shall soon be seeing, is that the corresponding nonlinear operator for fixed $t$ only depends on the first and second derivatives of $\varphi$ and not on the function itself. This will come in handy for the maximum and comparison principles in \S \ref{max and comp principles}.
\begin{remark}
	Since a solution of the equation for one of these functions naturally implies a solution for the equation of the other (by taking logarithm or exponentiating respectively), the equations for $u$ and $\varphi$ will be interchanged and handled with equality throughout. Note that $\varphi$ is well defined, since $u_0>0$ everywhere on $M^n$ and the a priori estimates of \S \ref{first order estimates} deliver us positive lower bounds on $u$ for $t>0$.
\end{remark}
\noindent In \S \ref{max and comp principles} as mentioned, we procure useful maximum and comparison principles for the acquired equations. These estimates are important for proving long time existence later on.\\
\\
We apply results of \cite{ger2, lady} in \S \ref{first order estimates} to prove existence of a solution, $\varphi$, to the adapted problem, then we go on to prove some first order estimates.\\
\\
To conclude the section, in \S \ref{evolution equations}, we study the so-called {\em evolution equations} of some quantities of interest on the evolving hypersurfaces. As we shall see, these are vital to the understanding of the behaviour of such solutions.\\
\\
We use the notation of \cite[\S 2.5]{ger2} throughout this section and for the duration of the paper. In particular, 
\[
	Q_{T^{\star}} \equiv [0,T^{\star}) \times M^n
\] denotes the {\em parabolic cylinder} of existence of the solution to \eqref{eq phi} and the class of function $H^{2+\beta, \frac{2+\beta}{2}}(\bar{Q}_{T^{\star}})$ to which (as we shall find out) the solution belongs is defined by
\begin{equation}\label{fcn space}
	\varphi \in H^{2+\beta, \frac{2+\beta}{2}}(\bar{Q}_{T^{\star}}) \iff \varphi(t,\cdot) \in C^{2,\beta}(M^n) \:  \wedge \: \varphi(\cdot,x) \in C^{1,\frac{\beta}{2}}([0,T^{\star})),
\end{equation} for $0<\beta<1$.
\subsection{posing the problem} \label{posing the problem}
The family of embeddings $x:Q_{T^{\star}} \equiv [0,T^{\star})\times M^n \rightarrow N^{n+1}$, of $M^n$ into the warped product $(N^{n+1} = \rel_+ \times_{\Id} M^n,\gamma)$ takes the form \[x:(t,p)\mapsto \left(u(t,p(t)),p(t)\right), \] so that in local coordinates $\lbrace y^i \rbrace$ on $M$ we can express the evolving hypersurface as 
\begin{equation} \label{graph hs}
	x(t,y^i(t)) = (u(t,y^i(t)), y^i(t)).
\end{equation}
We take a total time derivative of \eqref{graph hs} and substitute it into the original equation \eqref{eq} using \eqref{norm} to get 
\begin{equation} \label{radial x} 
	\dot{ x}^0 = \dt{ u} = \frac{\partial u}{\partial t} + u_i \dot{y}^i = H^{-1} \nu^0 = v^{-1} H^{-1}
\end{equation} 
for the radial component and 
\begin{equation} \label{tangential x}
	\dot{ x}^i = \dot{y}^i = -v^{-1}H^{-1}u^{-2}u^i
\end{equation}
for the components tangential to the level sets of the radial function, $\lbrace r = \text{const} \rbrace$. Plugging \eqref{tangential x} into \eqref{radial x} we infer
\begin{equation}
	\begin{split}
	\dot{u} & =  v^{-1}H^{-1}u^{-2} |Du|^2_{\sigma} +  v^{-1}H^{-1}\\
		  & = v^{-1}H^{-1}(1 + u^{-2}|Du|^2_{\sigma}) \\
		& = vH^{-1}
	\end{split}
\end{equation}
where the dot $\dot{}$ over $u$ has taken the place of $\partial / \partial t$. Multiplying both sides by $u^{-1}$ and using Corollary \ref{smc}, we come to the conclusion that in this setting, the original problem \eqref{eq} is equivalent to the nonlinear parabolic (of course this still needs to be checked) equation
\begin{equation}\label{eq phi}
\left\{
  	\begin{array}{l l}
    		\dot{\varphi} -  \left[ F(D\varphi, D^2 \varphi)\right]^{-1}  &= 0\\
    		\varphi_0  &= \varphi(0,\cdot)
  	\end{array} 
\right.
\end{equation}
with 
\begin{equation} \label{F}
	F \equiv v^{-2}(n - \sigma^{ij}\varphi_{ij}+ \varphi^i \varphi^j\varphi_{ij} v^{-2} ) = Huv^{-1}.
\end{equation}
\begin{remark}\label{F phi independent}
	Notice how $F = F(D\varphi, D^2 \varphi)$ is independent of $\varphi$ in an undifferentiated guise. This fact will come in useful in the near future.
\end{remark}

\begin{remark}
	The equation for the scalar quantity $u$ will also be of use in the future. It is given by
\begin{equation}\label{eq u}
\left\{ 
  \begin{array}{l l}
    \dot{u} -  H^{-1} v& =0 \\
    u_0 & =  u(0,\cdot),
  \end{array} \right.
\end{equation}
where $H = H(u,Du,D^2u)$ is expressed as
\[
	H = [uv]^{-1} \left( n- \left( u \Delta_g u - |Du|^2_g \right) \right).
\] The final term here, $|Du|^2_g$ is the norm, in terms of $g$, of the gradient of $u$.
\end{remark}


\subsection{maximum and comparison principles} \label{max and comp principles}
\begin{thm} \label{comp princ} Let $\varphi$ and $\tilde{\varphi}$ solve \eqref{eq phi} in $Q_{T^{\star}}$ with initial data $\varphi_0$ and $\tilde{\varphi}_0$ respectively. Then
	\begin{equation}	
		\inf_M (\varphi_0 - \tilde{\varphi}_0) \leq \varphi - \tilde{\varphi} \leq \sup_M (\varphi_0 - \tilde{\varphi}_0).
	\end{equation}
\end{thm}
\begin{proof}
Let $w \equiv (\varphi - \tilde{\varphi})\e^{\lambda t}$ for a small $\lambda$ to be chosen later on. Without loss of generality we assume \begin{equation}\label{sup phi}
	\sup_M (\varphi_0 - \tilde{\varphi}_0) \geq  0.
\end{equation} Let $0 < t_0$ be the first time the maximum of $w$ is attained and this occurs at the point $x_0 \in M^n$. At the {\em spacetime} point $(t_0,x_0)$ we have
\begin{align}
	0 \leq \dot{w} = & \lambda w + (\dot{\varphi} - \dot{\tilde{\varphi}}) \e^{\lambda t}\\
			= & \lambda w + ( F^{-1} - \tilde{F}^{-1}) \e^{\lambda t},
\end{align}
where $F\equiv F(D\varphi(t_0,x_0), D^2 \varphi(t_0,x_0))$ and $\tilde{F}\equiv F(D\tilde{\varphi}(t_0,x_0), D^2 \tilde{\varphi}(t_0,x_0))$. Since $w$ attains its spatial maximum at this point, we have
\begin{equation}
	D\varphi = D\tilde{\varphi} \quad \text{  and  } \quad  D^2 \varphi \leq D^2 \tilde{\varphi}.
\end{equation}
Using these facts along with \eqref{F}, we infer
\[F \geq \tilde{F}.\] Thus, at $(t_0, x_0)$ we conclude \[0\geq - \lambda w.\] If we choose $\lambda < 0$ it follows that 
\[
	\sup_{Q_{T^{\star}}} w = w(t_0,x_0) \leq 0
\] 
Letting $\lambda \nearrow 0$, bearing \eqref{sup phi} in mind, implies the right hand inequality. The left hand inequality is proved analogously.
\end{proof}
\begin{cor}\label{comp cor}
	Let $\varphi$ solve \eqref{eq phi} on $Q_T$. Then the following holds
	\begin{equation}\label{phi bound}
		\inf_{M^n} \varphi_0 \leq \varphi - \frac{t}{n} \leq \sup_{M^n} \varphi_0.
	\end{equation}
\end{cor}
\begin{proof}
	Define $\tilde{\varphi} = t/n$. Then $\tilde{\varphi}$ solves \eqref{eq phi}. Apply Theorem \ref{comp princ}.
\end{proof}

\begin{remark} \label{u exp bound}
	By exponentiation, this is equivalent to 
\begin{equation}\label{u bound}
	\inf_{M^n} u_0 \leq u \e^{-t/n} \leq \sup_{M^n} u_0,
\end{equation}
for the embedding function $u$, which solves \eqref{eq u}.
\end{remark}
\begin{remark}
	The solution $\tilde{\varphi}$ of \eqref{eq phi} in the proof of Corollary \ref{comp cor} corresponds to the solution of \eqref{eq} whereby the initial embedding is given by \[x_0:p\mapsto (1,p).\]
\end{remark}
\begin{remark}
As we shall be seeing shortly in \S \ref{convergence}, the estimate on $u$ in Remark \ref{u exp bound} alludes to the fact that if we were to introduce a factor of $\e^{-t/n}$, then the function $\tilde{u} = u\e^{-t/n}$ remains uniformly bounded from above and below. Thus we will be studying the convergence properties of \eqref{eq} for the corresponding rescaled embedding $\tilde{x}$ in \S \ref{rescaled hypersurfaces}.
\end{remark}

\subsection{short time existence and first order estimates} \label{first order estimates}
With the help of the parabolic theory \cite{ger2, lady, evans} we show that a solution to equation \eqref{eq phi} exists for a short time. Following this, as the heading suggests, we prove some first order estimates on $\varphi$.
\begin{lem}
Let $F = Huv^{-1}$ as defined in \eqref{F}. Then
\[ 
	a^{ij} \equiv -\frac{\partial F}{\partial \varphi_{ij}}
\]
is positive definite.
\end{lem}
\begin{proof}
	Since
\[
	F = v^{-2}(n - \sigma^{ij}\varphi_{ij}+ \varphi^i \varphi^j\varphi_{ij} v^{-2} ) = v^{-2}(n-u^2 g^{ij}\varphi_{ij}),
\]
we simply differentiate $F$ with respect to $\varphi_{ij}$ to find
\begin{equation}
	-a^{ij} = \frac{\partial F}{\partial \varphi_{ij}} = -  v^{-2}u^{2} g^{ij}.
\end{equation}
Since $u$ and $v$ are both positive and $g^{ij}$ is positive definite, having trace $n-1+(1+|D\varphi|^2)^{-1}$ with respect to $\sigma$, (see \cite[\S 2.2]{mullins}), we conclude that $a^{ij}$ is positive definite.
\end{proof}
\begin{remark}\label{aij pos def}
	Since, as remarked upon above, \[\trc_{\sigma} \{g^{ij}\} = n-1+v^{-2},\] and $v$ is uniformly bounded from above and below, we conlude that $F^{-1}$ is {\em uniformly parabolic} on $Q_{T^{\star}}$. We will also show this more explicitly later (see Theorem \ref{H bound}) where we derive an evolution equation for $F^{-1} \equiv \chi H^{-1}$, where $\chi$ is to be defined at the beginning of \S \ref{graphical nature}.
\end{remark}
\begin{thm}\label{short time existence}
	There exists a unique solution of class $H^{2+\beta,\frac{2+\beta}{2}}(Q_{T^{\star}})$ to \eqref{eq phi}, given $C^{2,\alpha}$ initial data, for $0<\beta<\alpha$, where $T^{\star}$ depends on $\beta$ and the initial data.
\end{thm}
\begin{proof}
	The prerequisites of \cite[Theorem 2.5.7]{ger2} are fulfilled, since we have 
\[
	F^{-2} a^{ij}> 0,
\] i.e. \eqref{eq phi} is uniformly parabolic on some time interval $[0,T^{\star})$.
\end{proof}
\begin{lem}
	Let $\varphi$ solve \eqref{eq phi} on $Q_{T^{\star}}$. Then the following holds
	\begin{equation}
		\inf_M \dot{\varphi}(0, \cdot) \leq \dot{\varphi} \leq \sup_M \dot{\varphi}(0, \cdot)
	\end{equation}
\end{lem}
\begin{proof}
	Differentiate \eqref{eq phi} with respect to $t$ to obtain
	\begin{equation}\label{dt eq phi}
		\ddot{\varphi} + F^{-2} \left(  - a^{ij} D_i D_j \dot{\varphi} + a^i D_i \dot{\varphi}  \right) = 0,
	\end{equation}
where $a^i = \partial F / \partial \varphi_i$. Applying the parabolic maximum principle (see \cite[\S 7.1 Theorem 8]{evans}) yields the result.
\end{proof}
\begin{lem}\label{Dphi exp}
Let $\varphi$ solve \eqref{eq phi} on $Q_{T^{\star}}$ and condition \eqref{ricci bound} hold. Then there exists a $\lambda > 0$, such that
\begin{equation}\label{Dphi exp eq}
	|D\varphi|^2_{\sigma}\e^{\lambda t} \leq \sup_{M^n} |D\varphi_0|^2_{\sigma} 
\end{equation}
on $Q_{T^{\star}}$.
\end{lem}
\begin{proof}
	We aim to derive an evolution equation for the left hand side of \eqref{Dphi exp eq} and apply the maximum principle to deduce the inequality. Firstly, let $w = \frac{1}{2}|D\varphi|^2$. We differentiate \eqref{eq phi} in the direction of $D\varphi$ to obtain
 \begin{align}
\notag		0 & = \sigma^{kl}\varphi_l D_k (\dot{\varphi} - F^{-1})\\
\label{Dphi(eq)}	  & = \dot{w} + F^{-2}\big( -a^{ij}\sigma^{kl}D_k D_j \varphi_i \varphi_l + a^iD_i w\big),
\end{align}
where $a^{i} \equiv \partial F/ \partial \varphi_i$, having used the following identities:
\[\dot{w} = \sigma^{kl}\varphi_l D_k \dot{\varphi}, \] 
\[D_i w = \sigma^{kl}D_i\varphi_k \varphi_l.\] 
Recall from \eqref{commutator form} that
\begin{equation}\label{curv phi}
	D_k D_j \varphi_i = D_k D_i \varphi_j = D_iD_k \varphi_j - \tensor{R}{_i_k_j^m}\varphi_m,
\end{equation} 
where the first equality is on account of the torsion free connection $D$ on $M^n$. We also observe
\begin{equation}\label{D2w}
	D_i D_j w = \sigma^{kl}D_iD_j\varphi_k \varphi_l + \sigma^{kl} D_j \varphi_k D_i\varphi_l.
\end{equation}
Subtracting and adding the quantity $F^{-2}a^{ij}\sigma^{kl} D_j \varphi_k D_i\varphi_l$ (i.e. adding a clever 0) to \eqref{Dphi(eq)} and plugging in \eqref{curv phi} and \eqref{D2w} we arrive at
\begin{equation} \label{wfin}
	\dot{w} + F^{-2} \big(  -a^{ij} D_iD_j w + a^{ij}\sigma^{kl}D_j\varphi_k D_i \varphi_l + a^{ij}\tensor{R}{_i_k_j_m}\varphi^m \varphi^k + a^i D_i w  \big) = 0.
\end{equation}
For second term inside the parenthesis we have:
\begin{equation}
	a^{ij}\sigma^{kl}D_j\varphi_{k} D_i \varphi_l \geq 0,
\end{equation}
since $\{a^{ij}\}$ and $\sigma^{ij}$ are both positive definite and symmetric (see \cite[\S 4]{ding}). Subtracting this term from \eqref{wfin} gives
\begin{equation}\label{wleq}
	\dot{w} +F^{-2} \big( -a^{ij}D_i D_j w + a^{ij}\tensor{R}{_i_k_j_m}\varphi^m \varphi^k+ a^iD_iw \big) \leq 0.
\end{equation}
Now multiply \eqref{wleq} by $e^{\lambda t}$ and define $\tilde{w} = we^{\lambda t},$
to give
\begin{equation}
	\dot{w}e^{\lambda t} +F^{-2} \big( -a^{ij}D_i D_j \tilde{w} + a^{ij}\tensor{R}{_i_k_j_m}\varphi^k \varphi^m e^{\lambda t} + a^iD_i\tilde{w} \big) \leq 0,
\end{equation}
and finally add and subtract $\lambda \tilde{w}$, leaving us with
\begin{equation}
	\dot{\tilde{w}} + F^{-2}\big(  -a^{ij} D_i D_j \tilde{w} + a^{ij} \tensor{R}{_i_k_j_m}\varphi^k \varphi^m e^{\lambda t} + a^i D_i \tilde{w} -\lambda F^2 \tilde{w} \big) \leq 0.
\end{equation}
Let $\mu$ be the smallest eigenvalue of $\{a^{ij}\}$ (in light of Remark \ref{aij pos def} we have $\mu >0$ everywhere on $Q_{T^{\star}}$). Then together with \eqref{ricci bound} 
\[
	a^{ij}\tensor{R}{_i_k_j_m} \varphi^k \varphi^m e^{\lambda t}\geq \mu \Ric_M(D\varphi)\e^{\lambda t} \geq \mu \delta |D\varphi|^2_{\sigma}\e^{\lambda t} = 2\mu \delta \tilde{w},
\] 
for a small $\delta >0$ (the existence of which is ensured in view of \eqref{ricci bound} and the compactness of $M^n$), which implies 
\[ 
	\dot{\tilde{w}} + F^{-2} \big( -a^{ij}D_i D_j \tilde{w} + a^iD_i \tilde{w}+ (2\mu \delta - \lambda F^2)\tilde{w}  \big) \leq 0.
\]
Thus, by applying the maximum principle (see \cite[\S 7.1 Theorem 9]{evans}), the lemma is proved with $\lambda \leq 2\mu\delta F^{-2}$.
\end{proof}

\subsection{the evolution equations}\label{evolution equations}
Assuming short time existence of the flow \eqref{eq}, we focus our attention on how the solution changes shape, and how the quantities of interest evolve under the flow. The evolution equations are paramount to the study of how solutions behave in the long term, or in the neighbourhood of singularities.\\
\\
Since we are dealing with quantities that live on the evolving hypersurfaces, we will be covariantly differentiating with respect to the metric $g$ of $M_t$ and the {\em pushforward basis} of $T_{x(p)}M_t$, thus with the operator $\nabla$. Furthermore, the derivative with respect to $t$, designated as both $\dot{}$ and $\dt{}$ interchangeably, now refers to the ambient covarient derivative in the direction of $\nu$ with a magnitude of $H^{-1}$, i.e. \[\dt{} = \frac{D}{\partial t}= H^{-1}\nablaN_{\nu}\]
\begin{lem}
	The unit normal vector $\nu$ evolves according to
	\begin{equation}\label{ev eq nu}
		\dot{\nu} = H^{-2} \nabla H = H^{-2} H^{;k}x_k
	\end{equation}
\end{lem}
\begin{proof} 
Since $\nu$ is of unit length and $\dt{}$ is compatible with the metric $\gamma$, we have
\begin{equation} \label{d of norm}
	0 = \dt{} 1 = \dt{} \gamma(\nu, \nu) = 2 \gamma\left( \dot{\nu} , \nu \right),
\end{equation}
or in other words, $\dot{\nu}$ resides exclusively in $TM_t$. We can express $\dot{\nu}$ thusly
\begin{equation}\label{nu dot}
	\dot{\nu} = g^{ik} \gamma(\dot{\nu}, x_i) x_k = - g^{ik} \gamma(\nu, \dot{x}_i) x_k.
\end{equation}
To see how the basis vector $x_i$ evolves, let $c:(-\epsilon,\epsilon)\longrightarrow M^n$ be a differentiable curve such that 
\[
	\ds{c} \bigg|_{s=0} = e_i. 
\] Per definition 
\[
	x_i = x_{\star} (e_i) = \ds{} \bigg|_{s=0} x(t,c(s)).
\]
This shows subsequently that
\begin{equation} \label{ev basis}
	\dot{x}_i = \frac{D}{\partial t} \frac{\partial}{\partial s} \bigg|_{s=0} x(t,c(s)) =  \frac{D}{\partial s}\bigg|_{s=0} \frac{\partial}{\partial t}  x(t,c(s)) = \frac{D}{\partial s}\bigg|_{s=0} (H^{-1}\nu)_{x(t,c(s))} =  (H^{-1}\nu)_i
\end{equation}
where we have used \cite[\S 3 Lemma 3.4]{do carmo} to reverse the order of differentiation (for details, see \cite{mullins}). Once again using \eqref{d of norm}, this time in the second equality of \eqref{nu dot} with $\de = (\nablaN_{x_i})^{\top} \equiv \nabla_i$, we conclude
\[
	\dot{\nu} = -g^{ik}\gamma(\nu, (H^{-1}\nu)_i) x_k = -g^{ik} (H^{-1})_i\gamma(\nu,\nu) x_k  = H^{-2} g^{ik} H_{;i} x_k.
\]
\end{proof}
\begin{remark} \label{commute}
	What equation \eqref{ev basis} alludes to is the fact that the flows of the vector fields $H^{-1}\nu$ and $x_i$ commute, i.e. 
\[ [H^{-1}\nu, x_{i}] = 0, \] or equivalently \[ \nablaN_{H^{-1}\nu}x_i = \nablaN_{x_i}(H^{-1}\nu) \equiv (H^{-1}\nu)_i.\]
\end{remark}
\begin{lem}
	The metric $g_{ij}$ evolves according to
	\begin{equation} \label{ev eq gij}
		\dot{g}_{ij} = 2H^{-1}h_{ij}.
	\end{equation}
\end{lem}
\begin{proof}
	Using \eqref{g def}, \eqref{weingarten eq}, \eqref{ev basis}, the symmetry of $g_{ij}$ and the fact that $\gamma(\nu,x_i) = 0$ we obtain
\[
	\dot{g}_{ij} = 2\gamma(\dot{x}_i, x_j) = 2\gamma((H^{-1}\nu)_i, x_j) = 2H^{-1}\gamma(\nu_i, x_j) = 2H^{-1}h_{ij}.
\]
\end{proof}
\begin{cor}
	The inverse metric fulfills the evolution equation
	\begin{equation} \label{ev eq inverse g}
		\dot{g}^{ij} = -2H h^{ij},
	\end{equation}
where $h^{ij} \equiv g^{ik} g^{jl} h_{kl}$.
\end{cor}
\begin{proof}
	Since $\delta^i_j = g^{ik} g_{kj}$ and $0 = \de\! \delta^i_j / \de\! t$, we have
	\[
		\dot{g}^{ik}g_{kj} = -g^{ik}\dot{g}_{kj}.
	\]
	Now contracting both sides with $g^{jl}$ and plugging in \eqref{ev eq gij} for $\dot{g}_{kj}$ we have
	\[
		\dot{g}^{il} = -2g^{ik}g^{jl}Hh_{kj} = -2Hh^{il}.
	\]
\end{proof}
\begin{lem}\label{ev eq measure}
	The evolution of the induced measure $\de\!\mu$ satisfies
	\begin{equation} \label{ev eq g}
		\dt{}(\de\!\mu) = \de\!\mu.
	\end{equation}
\end{lem}
\begin{proof}
	In coordinates, 
	\[
		\de\!\mu = \sqrt{g}\de\!x,
	\]
where $g = \dtm g_{ij}$. {\em Jacobi's formula} for taking the derivative of a determinant and the product rule give
\begin{equation}
	\dt{} \sqrt{g} = \frac{1}{2}( g)^{-\frac{1}{2}} g g^{ij} \frac{\partial g_{ij}}{\partial t} = \sqrt{g}H^{-1} g^{ij}h_{ij} = \sqrt{g}.
\end{equation}
\end{proof}
\begin{cor}
	For a surface evolving according to \eqref{eq}, as long as the flow exists, the following holds
	\[
		|M_t| = |M_0|\e^t,
	\] where $M_0$ is the initial surface.
\end{cor}
\begin{thm}
	The second fundamental form evolves according to
	\begin{equation}\label{ev eq hij}
		\begin{split}
			\dt{h_{ij}} = & (H^{-1})_{;ij}  + H^{-1}\left( h^m_i h_{mj} - \RN(\nu, x_i, \nu, x_j)\right)\\
						= &2 H_i H_j H^{-3} - H_{ij} H^{-2} + H^{-1}\left( h^m_i h_{mj} - \RN(\nu, x_i, \nu, x_j)\right).
		\end{split}
	\end{equation}
\end{thm}
\begin{proof}
	We utilise \eqref{gauss formula}, \eqref{ev eq nu}, Remark \ref{commute} and a myriad of additional tricks to derive the following 
\begin{equation}
	\begin{split}
		\dt{h_{ij}}  =&  - \dt{} \gamma(x_{ij}, \nu) \\
		= & - \gamma(\nablaN_{H^{-1}\nu} \nablaN_{x_j} x_i, \nu) - \gamma(x_{ij}, \dot{\nu})\\
		= &- \gamma(\nablaN_{x_j}\nablaN_{H^{-1}\nu} x_i, \nu) - H^{-1}\gamma(\RN(x_i,\nu)x_j,\nu)+\gamma(\nablaN_{[x_i, H^{-1}\nu]}x_j,\nu)\\
		= &-\gamma(\nablaN_{x_j}\nablaN_{x_i}(H^{-1}\nu), \nu) - H^{-1}\gamma(\RN(x_i,\nu)x_j,\nu)\\
		= &-(H^{-1})_{ij} - H^{-1}\gamma(\nu,\nu_{ij}) - H^{-1}\gamma(\RN(x_i,\nu)x_j,\nu)\\
		= &-(H^{-1})_{ij} + H^{-1} \left(h_i^k h_{kj} - \RN(x_i,\nu,x_j,\nu) \right )
	\end{split}
\end{equation}
\end{proof}

\begin{cor}
	The Weingarten map $h^{i}_j$  satisfies
\begin{equation}
	\begin{split}\label{ev eq weing}
		\dt{h^i_j} = & -(H^{-1})^i_j - H^{-1}\left(h^i_k h^k_j + \RN(x_k,\nu,x_j,\nu)g^{ki} \right )    \\
				 = &2 H^i H_j H^{-3} - H^i_{j} H^{-2} -  H^{-1}\left(h^i_k h^k_j + \RN(x_k,\nu,x_j,\nu)g^{ki} \right ) 
	\end{split}
\end{equation}
\end{cor}
\begin{proof}
Since
	\[
		\dot{h}^i_j = \dt{}(g^{ik}h_{kj}) = \dot{g}^{ik}h_{kj} + g^{ik} \dot{h}_{kj},
	\]
we can use \eqref{ev eq inverse g} and \eqref{ev eq hij} to yield the result.
\end{proof}
\begin{cor} \label{ev eq H}
The mean curvature evolves according to
	\begin{equation}
	\begin{split}
		\dt{H} = & -\Delta_{M_t}(H^{-1}) - H^{-1}\left( |A|_g^2 + \RicN (\nu) \right) \\
			= & H^{-2} \Delta_{M_t} H -2H^{-3} |\nabla H|^2_g - H^{-1}\left( |A|_g^2 + \RicN (\nu) \right).
	\end{split}
	\end{equation}
\end{cor}
\begin{proof}
	Take the trace of \eqref{ev eq weing}. 
\end{proof}
\begin{remark}
	By application of the maximum principle in the previous equation, in light of the negative sign on the $|A|^2$ term, we deduce that the mean curvature $H$ is uniformly bounded from above in terms of the ambient Ricci curvature and the initial embedding.
\end{remark}

\newpage
\section{long term behaviour}
In order to show that the flow \eqref{eq} exists for all times $t > 0$, we need to bound the mean curvature with an exponential factor, $H \e^{\lambda t}$, from above and below by positive constants. In the process, we shall also show that then the surfaces $M_t$ remain as graphs throughout the evolution (see Theorem \ref{graphical nature}) by deriving an evolution equation for the quantity $\chi = \gamma(X,\nu)^{-1}$ where $X = u \partial / \partial r$. \\
\\
In \S \ref{convergence} we rescale the embedded hypersurfaces with a factor of $\e^{-t/n}$, as already hinted at earlier on. \\
\\
Throughout this section we shall denote with $c$ and $C$ arbitrary positive constants that can, on occasion, vary from line to line (as contradictory as this sounds).

\subsection{graphical nature} \label{graphical nature}
The quantities that we will be dealing with reside on the evolving surfaces $M_t$ and as such, we will generally be dealing with the basis $\{x_i\}_{i=1,...,n}$ of $TM_t$. Indeed, all covariant derivatives will be carried out with respect to $\nabla(g)$ and indices will be raised and lowered with respect to either $g_{ij}$ or $\gamma_{ab}$. The context should be clear.
\begin{lem} Let $\chi = \gamma(X,\nu)^{-1}$ where $X = u \partial / \partial r$. Then
	\begin{itemize}
	\item[{\em (i)}]
		\begin{equation}
			X_{;i} = x_i \equiv u_i e_0 + e_i,
		\end{equation}
	\item[{\em (ii)}]
		\begin{equation}
			\gamma( X_{;ij}, \nu ) = \gamma( x_{ij}, \nu ) = -h_{ij}
		\end{equation}
	\item[{\em (iii)}]
		\begin{equation}
			\chi = vu^{-1}
		\end{equation}
	\end{itemize}
\end{lem}
\begin{proof} (i) is derived from \eqref{christoffel}
	\begin{equation}
		X_{;i} = X^a_b x^b_i = \left( \nablaN_{u_i e_0 + e_i} (u e_0) \right)^{\top} = u_i e_0 + u\tensor{\GammaN}{^a_i_0} e_a = u_i e_0 + \delta_i^k e_k = x_i.
	\end{equation}
(ii) follows (i) and \eqref{gauss formula}, and (iii) is easily seen as a result of $\gamma(\partial / \partial r, \nu) = v^{-1}$.
\end{proof}

The following evolution equation controls the graphical nature of the surfaces $M_t$.
\begin{lem} \label{graphical nature} Let $\chi  \equiv \gamma (X, \nu )^{-1}$ as above. Then $\chi$ satisfies the evolution equation
	\begin{equation}\label{graphical nature eq}
		\begin{split}
		\left( \dt{} - H^{-2}\Delta_{M_t} \right) \chi = -2\chi^{-1}H^{-2}|\nabla \chi|^2_g - \chi H^{-2} \left( |A|^2_g+ \RicN(\nu) \right).
		\end{split}
	\end{equation}
\end{lem}
\begin{proof} We begin by differentiating $\chi$ with respect to $t$
	\begin{equation}
		\dot{\chi} = -\chi^2 \big(  H^{-1} +  \gamma( X,H^{-2} \nabla H ) \big),
	\end{equation}
where we have used \eqref{ev eq nu}. Next we take first and second derivatives of $\chi$ in the directions $x_i$ and $x_j$
	\begin{equation}
		\chi_{;i} = -\chi^2 \big(  \gamma( x_i, \nu )+ \gamma (X, \nu_i ) \big),
	\end{equation}
and
	\begin{align} \label{chi ij}
\notag		\chi_{ij} = & 2\chi^3  \big (\gamma( x_j, \nu )+ \gamma( X, \nu_j ) \big ) \big (\gamma( x_i, \nu )+ \gamma( X, \nu_i ) \big ) \\
\notag		&-\chi^2 \big ( \gamma( x_{ij}, \nu ) +  \gamma( x_j, \nu_i )+ \gamma( x_i, \nu_j )+ \gamma( X, \nu_{;ij}) \big )\\ 
\notag		= &2\chi^{-1} \chi_i \chi_j - \chi^2 h_{ij} - \chi^2h^k_{i;j} \gamma( X,x_k )+ \chi h^k_i h_{kj}\\
\notag		= & 2\chi^{-1} \chi_i \chi_j - \chi^2 h_{ij} + \chi h^k_i h_{kj} \\  
		& - \chi^2 \big ( h^{;k}_{ij} \gamma( X,x_k ) + \RN(\nu, x_i, x_l, x_j) g^{kl}\gamma(X,x_k)\big ),
	\end{align}
where we have used \eqref{codazzi} contracted with $g^{kl}$ to get the curvature term. Using the fact that
\[ g^{kl}\gamma( X, x_k)x_l  = X^{\top} = X - \gamma( \nu,X )\nu,\]
we come to the conclusion that
\begin{align}\label{rn}
\notag	\RN( \nu, x_i, x_l, x_j) g^{kl}\gamma(x,x_k) = & \RN(\nu, x_i X, x_j) - \RN(\nu, x_i, \nu, x_j) \chi^{-1} \\
\notag							= &- \RN(\nu, x_i, \nu, x_j) \chi^{-1},
\end{align}
since $X = u \partial / \partial r$ and in light of  \eqref{curvature}. Contracting $\chi_{ij}$ with $H^{-2} g^{ij}$ and subtracting it from $\dot{\chi}$ delivers
\begin{equation}
	\left( \dt{} - H^{-2}\Delta_{M_t} \right) \chi = -2\chi^{-1}H^{-2}|\nabla \chi|^2_g - \chi H^{-2} \left( |A|^2_g+ \RicN(\nu) \right),
\end{equation}
as claimed.
\end{proof}

\begin{thm} \label{H bound}
	There exist positive constants $c$ and $C$, depending only on $u_0, Du_0, D^2 u_0$ and $n$ for which during the evolution process up to time $T$, the following holds 
	\begin{equation}
		c\leq H\e ^{t/n} \leq C.
	\end{equation}
\end{thm}
\begin{proof}
We shall prove only the lower bound here, the upper bound follows analogously with $\inf_M$ replacing $\sup_M$ in the following argument. We compute the evolution of $\chi H^{-1},$  using \eqref{graphical nature eq} and \eqref{ev eq H}. 
	\begin{equation}
		\begin{split}
			\dt{\ (\chi H^{-1})} = & \dot{\chi} H^{-1} - H^{-2}\chi \dot{H} \\
				= & H^{-3} \Delta \chi - 2 \chi^{-1} H^{-3}|\nabla \chi|^2_g + \chi H^{-2}\Delta(H^{-1})\\
				= & H^{-2} \left( H^{-1}\Delta\chi + \chi \Delta(H^{-1}) - 2\chi^{-1}H^{-1} |\nabla \chi|^2_g \right).
		\end{split}
	\end{equation}
Let $(t_0,x_0)\in M_t$ be an interior space-time point such that the following holds
	\begin{equation}
		(\chi H^{-1})(t_0,x_0) = \sup_{t\in [0,T)} \sup_{M_t} (\chi H^{-1}).
	\end{equation}
Then at $(t_0,x_0)$, 
\[ 0 = (\chi H^{-1})_i = \chi_i H^{-1} + \chi (H^{-1})_i,\]
and also
\[
	\dt{\ (\chi H^{-1})}(x_0) \geq 0
\]
Using these and the fact that \[\Delta(\eta\xi) = \Delta(\eta)\xi + \eta\Delta(\xi) + 2 \nabla \eta \cdot \nabla \xi,\] for arbitrary real-valued functions $\eta$ and $\xi$, we get
\[ 0 \leq \dt{\ (\chi H^{-1})} = H^{-2} \Delta(\chi H^{-1}) \leq 0,\] since $(t_0,x_0)$ is a maximum.
Adding $t\epsilon$ to $\chi H^{-1}$ and letting $\epsilon \searrow 0$ derives a contradiction and gives
\[
	\chi H^{-1} \leq \sup_{M_0} \chi H^{-1}
\] 
or in other words
\[
	\chi^{-1} H \geq \inf_{M_0} \chi^{-1} H.
\]
Since $\chi^{-1} = u/v$ and $v\geq 1$ we have
\[ 
	\inf_{M_0} \chi^{-1} H \leq \chi^{-1}H = uv^{-1}H \leq uH \leq H c^{\prime} \e^{t/n}
\]
where the final inequality comes from \eqref{u bound}. This subsequently offers
\[
	H\e^{t/n}  \geq c.
\]
\end{proof}
\begin{remark}
	Since $[F(D\varphi,D^2\varphi)]^{-1} = \chi H^{-1}$, the upper and lower bound from the previous proof shows that \eqref{eq phi} remains uniformly parabolic on $Q_{T^{\star}}$.
\end{remark}
\begin{thm}
	The solution of \eqref{eq} stays graphical as long as the solution exists.
\end{thm}
\begin{proof}
	We wish to derive a differential inequality for $\chi$ and apply a quasilinear comparison principle (see \cite[Corollary 2.5]{lieberman}) to show that $\chi$ is bounded from above. From \eqref{curvature} we have
	\[
		\RicN(e_0,e_a) = \RN_{0a} = \RN_{a0} = 0.
	\]
Together with \eqref{ric m} and the definition of the outward pointing unit normal $\nu$, \eqref{norm} we deduce
\begin{equation}
	\begin{split}
		\RicN(\nu) = &  \left[ v^{-1} u^{-2} \right]^2 \RN_{ik} u^i u^k \\
				= & \left[ uv \right]^{-2}\left(  \Ric_M(D\varphi) - (n-1)|D\varphi|^2_{\sigma} \right).
	\end{split}
\end{equation}
Using the prerequisite \eqref{ricci bound} along with the lower bound on $H$, the positivity of $\chi$ and $|A|^2$, and Lemma \ref{Dphi exp} we infer
\begin{equation}
	\left( \dt{} - H^{-2}\Delta_{M_t} \right) \chi \leq -2\chi^{-1}H^{-2}|\nabla \chi|^2_g + \chi H^{-2} c \e^{-\lambda t}.
\end{equation}
In view of Theorem \ref{H bound} we have 
\[H^{-2} \leq C \e^{2t/n}.\] 
In order to apply a quasilinear comparison principle (see \cite[Corollary 2.5]{lieberman}), we need a function that satisfies
\[ \dt{}f(t) \geq C \e^{(2/n-\lambda)t}f(t),\]
but since we only want to show the boundedness of $\chi$ and are not interested in a sharp estimate, a solution to the equation
\[\dt{}f(t)  = C \e^{2t/n}f(t),\] will suffice. Indeed $f(t)=C\e^{n\e^{2t/n}/2}$ does the job. We define the quasilinear parabolic operator $P$ by
\[Pw \equiv \left( \dt{} - H^{-2}\Delta_{M_t} \right) w + 2w^{-1} H^{-2}|\nabla w|_g^2 - H^{-2}c\e^{-\lambda t} w.\]
Then
\[ Pf > 0 \geq P\chi. \]
By applying a comparison principle \cite[Corollary 2.5]{lieberman}, we deduce that $\chi$ is bounded from above on $\bar{Q}_{T^{\star}}$ by $f$ and is thus bounded on compact intervals, proving the claim. 
\end{proof}

\subsection{rescaled hypersurfaces}  \label{rescaled hypersurfaces}
As alluded to previously, for later contemplations of convergence we must rescale the embedded hypersurfaces so that the solutions remain inside a compact set. In doing this, we can investigate the asymptotics without having to appeal to spatial infinity. We use a factor of $\e^{-t/n}$, procured from \eqref{u bound}. In  this section we investigate the a priori effect of rescaling.\\
\\
Consider the family of rescaled embeddings $\tilde{x}(t,\cdot)$ with \[\tilde{x}(t,p) = (\tilde{u}(t,p(t)),p(t)),\]
where $\tilde{u} = u\e^{-t/n}.$ Using this notation we observe the following rescaled quantities:
\begin{align}
		D\tilde{u} = &  Du \e^{-t/n} \\
		\dot{\tilde{u}} = & \dot{u}\e^{-t/n} - \frac{\tilde{u}}{n}\\
		\tilde{\varphi} = & \varphi - \frac{t}{n}\\
\label{resc dphi}		D\tilde{\varphi} = & D\varphi\\
		\dot{\tilde{\varphi}} = & \dot{\varphi} - \frac{1}{n}.
\end{align}
This shows that the induced metric on the rescaled evolving hypersurfaces satisfies
\begin{equation}
\tilde{g}_{ij} = \gamma(\tilde{x}_i, \tilde{x}_j) = \e^{-2t/n} g_{ij},
\end{equation}
while the inverse metric satisfies
\begin{equation}
	\tilde{g}^{ij} = \e^{2t/n} g^{ij}.
\end{equation}
Thus, the rescaled second fundamental form satisfies
\begin{equation}
	\tilde{h}_{ij} = h_{ij} \e^{-t/n},
\end{equation}
and the mean curvature as well as the Weingarten map rescale as expected
\begin{align}
	\tilde{h}^i_j = & h^i_j \e^{t/n} \\
	\tilde{H} = & H\e^{t/n} .
\end{align}
The rescaled hypersurfaces must now solve the equations: for $\tilde{\varphi}$
\begin{equation}\label{resc eq phi}
\left\{
  	\begin{array}{l l}
    		\dot{\tilde{\varphi}} -  \left[ F(\cdot,D\varphi, D^2 \varphi)\right]^{-1} +n^{-1} &= 0\\
    			\tilde{\varphi}_0  &= \varphi(0,\cdot),
  	\end{array} \right.
\end{equation}
in light of \eqref{resc dphi}, and for $\tilde{u}$
\begin{equation}\label{resc eq u}
\left\{
  	\begin{array}{l l}
    		\dot{\tilde{u}} -  v \tilde{H}^{-1} + \tilde{u}n^{-1}& =0\\
    			\tilde{u}_0 & = \tilde{u}(0,\cdot),
  	\end{array} \right.
\end{equation}
Both equations are obviously still uniformly parabolic on $Q_{T^{\star}}$.
\begin{lem}
	The rescaled quantity $|D\tilde{u}|$ decays exponentially fast to $0$.
\end{lem}
\begin{proof}
	From $\varphi =\log u$, we deduce 
\[
	|D\tilde{u}|^2 = |D\varphi|^2 \e^{-2t/n} u^2,
\]
and using Lemma \ref{Dphi exp} we have
\begin{equation}
	|D\tilde{u}|^2 \leq C \e^{-\lambda t - 2t/n }u^2.
\end{equation}
Taking the root from both sides gives
\[
	|D\tilde{u}| \leq C \e^{-\lambda t/ 2} u\e^{-t/n} \leq C^{\prime} \e^{-\lambda t /2},
\]
where, for the final inequality, we have once again used \eqref{u bound}. 
\end{proof} 
\subsection{convergence}\label{convergence}
Let us summarise the results we have gathered thus far. In \S \ref{posing the problem} we transformed the initial problem \eqref{eq} into two versions of a scalar problem, \eqref{eq phi} and \eqref{eq u} respectively. In \S \ref{first order estimates} we proved existence of a solution $\varphi$ to \eqref{eq phi}, which is equivalent to existence of a solution $u = \e^{\varphi}$ to \eqref{eq u}, on the parabolic cylinder $Q_{T^{\star}}$ for $0<T^{\star} \leq \infty$. \\
\\
We had derived a priori estimates for $\varphi - t/n (\equiv \tilde{\varphi})$ and $u\e^{-t/n}(\equiv \tilde{u})$ in \S \ref{max and comp principles}, then we went on to bound $\dot{\varphi}$ from above and below with its initial data and $|D\varphi|^2_{\sigma}$ from above with an exponentially decaying factor (which is equivalent to $1 < v \leq 1+C\e^{-\lambda t}$).\\
\\
We showed that the {\em starshapedness} of the solution as a hypersurface is maintained throughout the evolution in \S \ref{graphical nature} and then proved bounds on $H\e^{t/n}( \equiv \tilde{H})$ from above and below. These bounds show that the nonlinear operator $G=F^{-1}= v[uH]^{-1}$  of \eqref{eq phi} is uniformly bounded from above and below up to time $T^{\star}$.\\
\\
With these bounds, we then rescaled the surfaces in \S \ref{rescaled hypersurfaces} to acquire a solution $\tilde{u}$ to the uniformly parabolic equation \eqref{resc eq u} which remains inside a compact set during its existence. We also showed that $|D\tilde{u}|$ decays exponentially fast.\\
\\
Now we wish to achieve higher order estimates for $\varphi$ independent of $t$. The notation we use throughout this section in terms of norms and spaces can be found in \cite[\S 2.5]{ger2} and also in \cite{lady}.
\begin{thm} 
	Let $\varphi$ solve \eqref{eq phi} in $Q_{T^{\star}}$. Then for all $\alpha \in (0,1)$, $\delta \in (0,T^{\star})$ and $m\geq 1$,
	\begin{equation}\label{phi bound m}
		|\varphi|_{m+\alpha,Q_{\delta,T^{\star}}} \leq C,
	\end{equation}
\end{thm}
where $C$ does not depend on $t$ and $Q_{\delta,T^{\star}}\equiv [\delta,T^{\star}) \times M^n$.
\begin{proof} Since $M^n$ is smooth and $F^{-1}$ is $C^{\infty}$ in its arguments, we can apply \cite[Theorem 2.5.10]{ger2} to our situation, which states that under the above prerequisites, $\varphi$ is of class $H^{2+m+\alpha, \frac{2+m+\alpha}{2}}(Q_{\delta,T^{\star}})$ for $m\geq 1$. We differentiate \eqref{eq phi} in the direction of $e_k$, using \eqref{commutator form} to get
\begin{equation}
	0=e_k\left( \dot{\varphi} - F^{-1}\right) = \dot{\varphi}_k + F^{-2}\left\{-a^{ij}D_i D_j\varphi_{k} + a^{ij}[ {^{\sigma}\tensor{R}{_i_k_j^m}} ]\varphi_m + a^i D_i\varphi_k\right\}.
\end{equation}
Since $a^{ij} = [u/v]^2g^{ij} = v^{-2}(\sigma^{ij}-\varphi^i \varphi^j v^{-2})$, using again the symmetries of $R_{ijkl}$, this transforms into
\begin{equation}
	 0 = \dot{\varphi}_k + F^{-2} \left\{ - a^{ij}D_i D_j \varphi_k+ v^{-2}\Ric_M(D\varphi, e_k) + a^i D_i \varphi_k \right\} 
\end{equation}
In addition to this, we once again differentiate \eqref{eq phi} with respect to $t$ (see proof of Lemma \ref{Dphi exp}) to achieve
\begin{equation}
	\begin{split}
		\ddot{\varphi} + F^{-2}\left\{-a^{ij} \dot{\varphi_{ij}} + a^{i}\dot{\varphi_{i}}\right \} & = 0\\
		\dot{\varphi}_1 + F^{-2} \left\{ - a^{ij}D_i D_j \varphi_1+ v^{-2}\Ric_M(D\varphi, e_1) + a^i D_i \varphi_1 \right\} &=0\\
			\vdots \hspace{200pt}	\vdots & \\
		\dot{\varphi}_n + F^{-2} \left\{ - a^{ij}D_i D_j \varphi_n+ v^{-2}\Ric_M(D\varphi, e_n) + a^i D_i \varphi_n \right\} &=0,
	\end{split}
\end{equation}
which is a system of $(n+1)$ linear parabolic equations of second order for the first derivatives of $\varphi$ on the interval $[\delta,T^{\star})$. We may now apply \cite[Theorem 5.1]{lady}, which delivers uniform bounds on $|\varphi_k|_{2+\alpha,Q_{\delta,T^{\star}}}$ for $k=1,\dots,n$. Using an induction argument achieves the result.
\end{proof}
\begin{lem}
	$|D^2 \varphi|$ decays exponentially fast to $0$.
\end{lem}
\begin{proof}
	An interpolation inequality due to Hamilton \cite[Cor. 12.7]{ham} reads
\[
	\int |D^2 \varphi|^2 d\mu \leq C \left \{ \int |D^3 \varphi|^2 d\mu \right\}^{\frac{2}{3}} \left\{  \int |D\varphi|^2 d\mu \right\}^{\frac{1}{3}},
\]
and in light of \eqref{phi bound m} and the compactness of $M^n$ we infer
\[
	\int |D^2 \varphi|^2 d\mu \leq C \e^{-\lambda t}.
\]
\end{proof}
\begin{remark}
	This result implies 
\begin{equation}\label{D2 u bound}
	|D^2 \tilde{u}| \leq C \e^{-\beta t},
\end{equation}
for a constant $\beta>0$ as one can easily see, using the fact that 
\[
	D^2 \varphi =  u^{-1}D^2 u - u^{-2}Du \otimes Du,
\] together with our known estimates, provides
\[
|D^2 \tilde{u}| \leq c u^{-1}|D^2u| = c\left( |D^2 \varphi| + |D\varphi|^2 \right) \leq C e^{-\beta t}.
\]
\end{remark}

\begin{cor}
	The non-linear parabolic equation \eqref{eq phi} has a solution for all times $t>0$.
\end{cor}
\begin{proof}
	Arguing as in \cite[Lemma 2.6.1]{ger2}, we assume $T^{\star}< \infty$ is the maximal time of existence. Let $m\geq 0$. With the uniform bounds we have obtained for $|\varphi|_{2+m+\alpha,Q_{\delta,T^{\star}}}$, by applying Arzela-Ascoli, we see that there exists a sequence $t_j \rightarrow T^{\star}$ for which
\begin{equation}
	\varphi(t_j,\cdot) \xrightarrow[j \rightarrow \infty]{}\bar{\varphi},
\end{equation}
in $C^{2+m,\beta}(M^n)$, for $0<\beta<\alpha$, with $\bar{\varphi}\in C^{2+m,\alpha}(M^n)$. Then choosing $\bar{\varphi}$ as initial value in Theorem \ref{short time existence}, for which the flow exists at least up to time $\epsilon_1>0$, we deduce that there exists an open set $U(\bar{\varphi}) \subset C^{2+m,\beta}(M^n)$, for whose members the flow also exists up to time $\epsilon_1$. Now choosing $j_0$ so large that $\varphi(t_{j_0}, \cdot) \in U(\bar{\varphi})\cap C^{2+m,\alpha}(M^n)$ and that $t_{j_0}+ \epsilon_1 > T^{\star}$, we extend the maximal time of existence by once again applying the existence Theorem \ref{short time existence}, this time with initial data $\varphi(t_{j_0},\cdot)$, thus contradicting the maximality of $T^{\star}$. Therefore, a solution exists for all $t>0$.
\end{proof}

\subsection{asymptotics}
From \eqref{D2 u bound}, since \eqref{u bound} shows that $\tilde{u}$ stays inside a compact set during evolution and bearing Lemma \ref{Dphi exp} in mind, we conclude that $\tilde{u}$ is asymptotically constant, we shall call this constant $r_{\infty}$. To achieve a value for $r_{\infty}$ we will need to make some observations. 
\begin{lem}
	The rescaled quantities satisfy
\begin{itemize}
	\item[{\em (i)}]
	\begin{equation}
		\tilde{g}_{ij}\xrightarrow[t \rightarrow \infty]{} r_{\infty}^2 \sigma_{ij}
	\end{equation}
	\item[{\em (ii)}] 
	\begin{equation}
		\tilde{h}^i_j \xrightarrow[t \rightarrow \infty]{} r_{\infty}^{-1} \delta^i_{j}
	\end{equation}
	\item[{\em (iii)}]\label{rescaled volume}
	\begin{equation} \label{ev eq tilde gij}
		\dot{\tilde{g}}_{ij} = 2\tilde{H}^{-1} \tilde{h}_{ij} - 2n^{-1} \tilde{g}_{ij}
	\end{equation}
	\item[{\em (iv)}]
	\begin{equation}
		\dt{}(\de \! \tilde{\mu})=  0 
	\end{equation}
\end{itemize}
\end{lem}
\begin{proof} 
	Using the already procured estimates we conclude for (i) and (ii)
\[
	\tilde{g}_{ij} = \e^{-2t/n}g_{ij} = \tilde{u}^2 (\sigma_{ij} + \varphi_i \varphi_j) \xrightarrow[t \rightarrow \infty]{}  r_{\infty}^2 \sigma_{ij},
\]
\[
	\tilde{h}^i_j = [\tilde{u}v]^{-1} (\delta^i_j - \sigma^{ik} \varphi_{kj} + \varphi^{i}\varphi^{k}\varphi_{kj}v^{-2}) \xrightarrow[t \rightarrow \infty]{}  r_{\infty}^{-1} \delta^i_{j}.
\]
(iii) can be obtained from the evolution equation \eqref{ev eq gij} and the product rule and (iv) is simply a consequence of (iii) and the evolution equation \eqref{ev eq g} (notice how contracting \eqref{ev eq tilde gij} with $\tilde{g}^{ij}$ gives $0$). 
\end{proof}
\noindent What Lemma \ref{rescaled volume}(iii) alludes to is that during the evolution process of the surfaces, the volume stays constant. Hence, $r_{\infty}$ depends only on the initial embedding $u_0$.
\begin{cor}
	\begin{equation}
		r_{\infty} =  \left[ \frac{|M_0|}{|M^n|}\right]^{1/n}.
	\end{equation}
\end{cor}
\begin{proof}
	Equation \eqref{metric alpha} implies that the volume element on the level sets of the radial function $r$ in $N^{n+1}$ and the volume element on $M^n$, betokened $d\mu^{\alpha}$ and $d\mu^{\sigma}$ respectively, are related by
\[
	\de\!\mu^{\alpha} = r^n \de\!\mu^{\sigma},
\]
since, in local coordinates
\[
	\sqrt{\dtm\{\alpha_{ij}\}} = r^n \sqrt{\dtm \{\sigma_{ij} \}}.
\]
Thus, we have 
\[
	|M_0| = r_{\infty}^n |M^n|.
\]
\end{proof}


\newpage
\appendix

\newcommand{\Addresses}{{
  \bigskip
  \footnotesize

  T.~Mullins, \textsc{R3S, Fraunhofer IZM, Gustav-Meyer-Allee 25, 13355 Berlin, Germany }\par\nopagebreak
  \href{mailto:thomas.mullins@tu-berlin.de}{thomas.mullins@tu-berlin.de} 
}
}
\Addresses
\


\begin{thebibliography}{MSY}

\bibitem{cns} L. Cafferelli, L. Nirenberg \& J. Spruck {\em The Dirichlet problem for nonlinear second order elliptic equations, III; Functions of the eigenvalue of the Hessian}, Acta Math. {\bf155 } (1985) 261-301

\bibitem{do carmo} M. do Carmo, {\em Riemannian geometry. Translated from Portugese by Francis Flaherty}, Birkh\"auser Boston 1992

\bibitem{do carmo df} M. do Carmo, {\em Differential Forms and Applications}, Springer-Verlag Berlin Heidelberg, 1994

\bibitem{clutterbuck} J. Clutterbuck, {\em Parabolic equations with continuous initial data}, arXiv:math/0504455v1 [math.AP], 2004

\bibitem{ding} Qi Ding, {\em The inverse mean curvature flow in rotationally symmetric spaces}, Chin. Ann. Math. {\bf 32}B(1) (2011) 27-44
\bibitem{evans} L.C. Evans, {\em Partial Differential Equations}, Graduate studies in mathematics, Volume 19, American Mathematical Society, 1998

\bibitem{ger1} C. Gerhardt, {\em Flow of nonconvex hypersurfaces into spheres }, J. Differential Geometry {\bf 32} (1990) 299 - 314

\bibitem{ger2} C. Gerhardt, {\em Curvature Problems}, Series in Geometry and Topology {\bf 39}, International Press 2006

\bibitem{ham} R. Hamilton, {\em Three-manifolds with positive Ricci curvature}, J. Differential Geom. {\bf 17}, Number 2 (1982), 255-306

\bibitem{mcf hui} G. Huisken {\em Flow by mean curvature of convex surfaces into spheres}, J. Differential Geometry {\bf 20} (1984) 237 - 266

\bibitem{geom ev eq} G. Huisken \& A. Polden, {\em Geometric evolution equations for hypersurfaces}, Calc. of Var. and Geom. Ev. Problems, Lecture Notes in Mathematics {\bf 1713} (1999) 45-84

\bibitem{jost} J. Jost, {\em Riemannian geometry and geometric analysis}, Universitext, Springer-Verlag, Berlin, 2002

\bibitem{krylov} N.V. Krylov, {\em Nonlinear elliptic and parabolic equations of the second order},  Reidel, Dordrecht, 1987

\bibitem{lady} O. A. Ladyzenskaja, V. A. Solonnikov \& N. N. Ura'ceva {\em Linear and Quasi-linear Equations of Parabolic Type. Translated from the Russian by S.Smith}, Translations of Mathematical Monographs. {\bf 23} American Mathematical Society, Rhode Island, 1968

\bibitem{lieberman} G.M. Lieberman, {\em Second Order Parabolic Differential Equations}, World Scientific, 1996 

\bibitem{marquardt} T. Marquardt, {\em The inverse mean curvature flow for hypersurfaces with boundary}, Dissertation, Freie Universit{\"a}t Berlin 2012

\bibitem{mullins} T. Mullins, {\em On minimal submanifolds}, Freie Universit{\"a}t Berlin 2012

\bibitem{petersen} P. Petersen, {\em Warped products}, Petersen's UCLA webpage, http://www.math.ucla.edu/~petersen/warpedproducts.pdf

\bibitem{penrose} R. Penrose, {\em The Road to Reality}, Jonathan Cape, London, 2004

\bibitem{simons} J. Simons, {\em Minimal varieties in Riemannian manifolds}, Ann. of Math. {\bf 88} (1968) 62 - 105

\bibitem{spivak} M. Spivak, {\em A Comprehensive Introduction to Differential Geometry, Volume Three}, Publish or Perish Inc., 1990

\bibitem{urbas} J.I.E. Urbas, {\em On the expansion of starshaped hypersurfaces by symmetric functions of their principal curvatures}, Math. Zeitschrift {\bf 205} (1990) 355-372

\bibitem{wald} R.M. Wald, {\em General Relativity}, The University of Chicago Press, Chicago, 1984

\end{thebibliography}
\end{document}